\documentclass[11pt]{article}

\usepackage{amsfonts}
\usepackage{xr}
\usepackage{graphicx,psfrag}
\usepackage{amsmath}
\usepackage{color,soul,textcomp}
\usepackage[active]{srcltx}

\def\2{C^{1,2}(\R\times\R^N)}

\def\C{\mathcal{C}}

\def\R{\mathbb{R}}

\def\tilde{\widetilde}
\def\.{\cdot}

\newlength{\textlarg}

\newcommand{\be}{\begin{equation}}
\newcommand{\ee}{\end{equation}}
\newcommand{\baa}{\begin{array}}
\newcommand{\eaa}{\end{array}}
\newcommand{\ba}{\begin{eqnarray}}
\newcommand{\ea}{\end{eqnarray}}

\newtheorem{thm}{\bf Theorem}[section]
\newtheorem{lem}[thm]{\bf Lemma}

\newtheorem{defi}[thm]{\bf Definition}
\newtheorem{rmq}[thm]{\bf Remark}

\newenvironment{proof}[1][Proof]{\noindent\textit{#1.} }{\hfill \rule{0.5em}{0.5em}}

\newenvironment{formula}[1]{\begin{equation}\label{eq:#1}}
                       {\end{equation}\noindent}
\def\Fi#1{\begin{formula}{#1}}
\def\Ff{\end{formula}\noindent}

\begin{document}
\author{Thomas Giletti$^{\hbox{\small{a}}}$ \\
\footnotesize{$^{\hbox{a }}$Graduate School of Mathematical Sciences, University of Tokyo, Komaba, Tokyo 153-8914, Japan}}
\date{}
\title{Convergence to pulsating traveling waves with minimal speed\\ in some KPP heterogeneous problems}

\maketitle

\begin{abstract}
The notion of traveling wave, which typically refers to some particular spatio-temporal connections between two stationary states (typically, entire solutions keeping the same profile's shape through time), is essential in the mathematical analysis of propagation phenomena. They provide insight on the underlying dynamics, and an accurate description of large time behavior of large classes of solutions, as we will see in this paper. For instance, in an homogeneous framework, it is well-known that, given a fast decaying initial datum (for instance, compactly supported), the solution of a KPP type reaction-diffusion equation converges in both speed and shape to the traveling wave with minimal speed. The issue at stake in this paper is the generalization of this result to some one-dimensional heterogeneous environments, namely spatially periodic or converging to a spatially periodic medium. This result fairly improves our understanding of the large-time behavior of solutions, as well as of the role of heterogeneity, which has become a crucial challenge in this field over the past few years.
\end{abstract}



\section{Introduction}

We consider the following spatially heterogeneous reaction-diffusion equation in $(0,+\infty) \times \mathbb{R}$:
\begin{equation}\label{eqn:eqRD}
\partial_t u = \partial_{xx} u + f(x,u).
\end{equation}
Throughout this work, we will assume that $f=f(x,u)$ is locally Lipschitz-continuous in $u$, $$f(x,0) \equiv 0$$ and $f$ is of class $\C^1$ in a neighborhood of $u=0$ uniformly with respect to $x$, so that in particular the derivative $ \partial_u f (x,0)$ is well-defined. 

Moreover, $f$ is of the heterogeneous KPP type in the following sense
\begin{equation}\label{eqn:KPP_hyp}
\begin{array}{c}
\displaystyle \forall s_2 > s_1 >0, \ \inf_{x \in \R} \left(\frac{f(x,s_1)}{s_1} - \frac{f(x,s_2)}{s_2} \right)>0, \vspace{5pt}\\
\displaystyle \exists ! p(x) \in L^\infty (\mathbb{R}, \mathbb{R}_+^*) \ \mbox{ such that } \  \liminf_{x \rightarrow +\infty} p(x) >0 \ \mbox{ and } \ \partial_{xx} p + f(x,p) =0.
\end{array}
\end{equation}
For instance, in a biological context where $u$ may refer to the population density of some species, the first hypothesis means that the growth rate of this species decreases as soon as its density increases, which is the main wanted property of KPP type nonlinearities. In other words, some saturation effect appears as soon as the population is positive. The second hypothesis states the existence of some positive and bounded stationary state, toward which propagation will occur. Our assumption on the right behavior of $p$ means that it stays away from the other equilibrium state zero ahead of the domain, that is in the direction of the propagation.

A typical $f$ which satisfies these hypotheses is the logistic growth law $f(x,u) = \mu (x) u(p(x)-u)$, where $\mu$ and $p$ are regular, positive and bounded functions. In such a situation, one will observe spreading from the state 0 to the state $p$ in the associated Cauchy problem with nonnegative and non trivial initial datum. Our goal in this paper will be to look at the spreading properties of such solutions of the Cauchy problem to the right, when $f$ is either spatially periodic, or strongly converges to a spatially periodic nonlinearity ahead of the domain. We will not only address the question of the spreading speed, but also the problem of the shape of the profile of propagation in large time, which will converge to a pulsating travelling wave whose definition will be recalled below.\\

We will first begin with the periodic case, which as we will see has been much studied in the past decade. In such a setting, the notion of pulsating traveling wave is well defined, and allows to describe very accurately the large time behavior of solutions for large classes of initial data, even fast decaying ones, which was an awaited result (see also the parallel work \cite{HNRR}). We will then look, as mentioned above, at some particular but strongly heterogeneous problem, where $f$ only exponentially converges (with a high enough rate that we will precise later) to some periodic nonlinearity as $x \rightarrow +\infty$. As one may expect, we will show that the large time behavior of solutions of the Cauchy problem is still given by traveling waves of the limiting periodic problem. This is, up to the author's knowledge, an entirely new result, and one of the first stating the convergence of the profile of a solution in a non periodic environment, the study of such heterogeneities being a challenging but essential open problem in the mathematical area of reaction-diffusion equations and propagation phenomena.

\subsection{In periodic media}

We will first consider the periodic case, that is $f$ is $L$-periodic with respect to the $x$-variable. The KPP periodic reaction-diffusion equation has been well-studied in the past decade \cite{BHR-I,BHR-II,Weinberger}, in particular the existence of pulsating traveling waves. This notion generalizes the notion of traveling wave in the homogeneous environment, which meant particular solutions of a reaction-diffusion equation moving through the domain with both a constant speed and constant profile. The notion of pulsating traveling wave is slightly more intricate, as the shape of the profile fluctuates (or ``pulses") due to the periodic heterogeneity. We recall it rigorously below:\begin{defi}\label{def:puls}
A \textbf{pulsating traveling wave solution} (or \textbf{pulsating
traveling front}) of \eqref{eqn:eqRD} connecting $0$ to $p(x)>0$ is an entire solution $u$ satisfying, for some $T >0$,
$$
u(t,x-L)=u(t+T,x),
$$
for any $x\in \R$ and $t \in \R$, along with the asymptotics
$$
u(-\infty,\cdot)=0 \; \mbox{ and } \; u(+\infty,\cdot)
=p(\cdot),
$$
where the convergence is understood to hold locally uniformly
in the space variable. The ratio $c:=\frac{L}{T} >0$ is called
the \textbf{average speed} (or simply the~\textbf{speed}) of
this pulsating traveling wave.
\end{defi}
\begin{rmq}
One can easily check that, for any $c>0$, $u(t,x)$ is a pulsating
traveling wave connecting $0$ to $p$ with speed $c$ if and only
if it can be written in the form $u(t,x) = U (x-ct,x)$, where~$U(z,x)$ satisfies
\begin{equation*}
\begin{array}{c}
 U (\cdot,x+L ) \equiv U (\cdot,x), \vspace{3.5pt}\\
 U (+\infty, \cdot) =0 \ \mbox{ and }  U (-\infty,\cdot)
 =p(\cdot),
\end{array}
\end{equation*}
along with the following equation that is
equivalent to \eqref{eqn:eqRD}:
\[
(\partial_x+\partial_z)^2U+cU_z+f(x,U)=0,\ \ \forall (z,x)\in\R^2.
\]
\vspace{-15pt}
\end{rmq}
Roughly speaking, this definition means that a pulsating traveling wave also moves through the domain with a constant speed, while its profile is no longer constant but periodic in time instead.

Of course, this definition requires $p$ to be a periodic stationary solution of \eqref{eqn:eqRD}. As we will see now, this is the case of the function $p$ whose existence we stated in our KPP hypothesis~\eqref{eqn:KPP_hyp}. Indeed, let the following principal eigenvalue problem:
\begin{equation}\label{principal-eigen}
\left\{
\begin{array}{rcl}
\displaystyle -\partial_{xx} \phi_\lambda + 2 \lambda \partial_x \phi_\lambda - \frac{\partial f}{\partial u} (x,0 ) \phi_\lambda & = & \mu (\lambda) \phi_\lambda \ \mbox{ in } \R, \vspace{3pt}\\
\phi_\lambda >0 \mbox{ and  $L$-periodic}.
\end{array}
\right.
\end{equation}
For our KPP hypothesis~\eqref{eqn:KPP_hyp} to hold, we need to assume that 0 is linearly unstable, that is $\mu (0) < 0$ (otherwise, there can be no positive and bounded stationary solution of \eqref{eqn:eqRD} \cite{BHR-I}); moreover, it is then known that the positive bounded stationary solution $p$ of \eqref{eqn:eqRD} is unique and periodic with respect to $x$ \cite{BHR-I}. In other words, our problem does reduce to the periodic one, although we did not assume a priori the periodicity of our positive equilibrium state. 

One can then define 
$$c^* := \min_{\lambda >0} \frac{\lambda^2 - \mu (\lambda)}{\lambda} >0.$$
Note that for $c=c^*$, the equation $\lambda^2 - \mu (\lambda) - c^* \lambda =0$ admits a unique solution $\lambda^* >0$. For $c>c^*$, the equation $\lambda^2 - \mu (\lambda) - c \lambda =0$ admits two positive solutions, one larger than $\lambda_*$ and one smaller than $\lambda_*$ (see Lemma~2.1 in \cite{H2008}).

It is known that there exists a pulsating traveling wave with speed $c$ if and only if $c \geq c^*$~\cite{BHR-II}. Moreover, for each $c$, this pulsating front is unique up to shift in time and increasing in time~\cite{HR}. In the following, we will denote by $U_c (t,x)$ the unique front with speed $c$ such that
$$U_c (0,0)= \frac{p(0)}{2}.$$
It was also shown in~\cite{H2008} that for any $c > c^*$, there exists some constant $B(c)>0$ such that~$U_c$ has the following asymptotic behavior on the right:
$$U_c (t,x+ct) \sim B(c) e^{-\lambda_c x} \phi_{\lambda_c} (x+ct) \ \mbox{ as } x \rightarrow +\infty \mbox{ uniformly in } t\in \R,$$
where $0<\lambda_c <\lambda_*$ is the smallest $\lambda$ such that $\lambda_c^2 - \mu (\lambda_c) - c\lambda_c =0$, and $\phi_{\lambda_c}$ the associated principal eigenfunction, normalized so that $\max \phi_{\lambda_c} =1$.\\

Those speeds, in particular the minimal speed $c^*$, are known to play an important role in spreading dynamics in the Cauchy problem associated with equation~\eqref{eqn:eqRD}. For instance, it is known that a compactly supported initial datum will spread with the speed $c^*$ in both directions~\cite{BHNadin,Weinberger}. That is, if we look in a moving frame with speed less than $c^*$, we see $u$ go to $p$ as $t \rightarrow +\infty$, while if we look in a moving frame with speed larger than $c^*$, we see $u$ go to $0$ in large time.

However, convergence of the profile of solutions of the Cauchy problem to that of a pulsating traveling wave is much less known. Most results in the past few years were concerned with initial data which behave similarly to a pulsating traveling wave as $x \rightarrow +\infty$, for instance initial data which can be trapped between two shifts of a pulsating front \cite{BMR08,BMR11}. The issue at stake in this paper will be the more physically relevant case of compactly supported or fast decaying initial data, which are expected to converge to the pulsating wave with minimal speed. A first step was proven as a corollary in \cite{DGM}, that is the case of an Heaviside type initial datum. By analogy with the proof of Lau in the homogeneous case \cite{Lau85} and for fast decaying initial data, this will lead us to prove the following theorem:

\begin{thm}\label{th:CV_p}
Assume that $f$ is KPP and periodic, and denote by $U_{c^*}$ the pulsating front with minimal speed. Let some differentiable initial datum $0 \leq u_0 \leq p$ such that
$$|u_0(x)| + |u'_0 (x) | \leq A e^{-\lambda_* x},$$
for some $A>0$ and with $\lambda_{*}$ the unique solution of $\lambda_*^2 - c^* \lambda_* = \mu (\lambda_* )$.

Then there exists $m(t)=o(t)$ such that the associated solution of the Cauchy problem satisfies
$$\|u(t,\cdot)-U_{c^*} (t-m(t),\cdot) \|_{L^\infty (\mathbb{R}^+)} \rightarrow 0$$
as $t \rightarrow +\infty$.
\end{thm}
\begin{rmq}
Note that it is immediate, by symmetry, that a similar result holds when looking at propagation to the left of the domain.

Moreoever, the regularity assumption on $u_0$ could actually be weakened, as our proof will only need the estimate on $|u_0 ' (x)|$ to be satisfied for $x$ large enough.
\end{rmq}
This result is also proven, at least for compactly supported initial data, in a parallel work by Hamel, Nolen, Roquejoffre and Ryzhik \cite{HNRR}. Furthermore, they provide an estimate on the shift~$m(t)$, namely that $m(t)=\frac{3}{2 \lambda_* } \log t + O(1)$. This logarithmic shift had only been proven before in the homogeneous case using probabilistic techniques~\cite{Bramson83}.

Our method will roughly show that the convergence holds as soon as the average spreading speed of $u$ is getting close to $c^*$, that is as soon as the position of the profile is close to $c^*t $, up to some $o(t)$. This fact will be needed in the second part of this paper, where we will consider perturbed periodic equations.

\subsection{In ``close-to-periodic" media}

In this section, we no longer assume that $f$ is periodic in its $x$-variable. We will now make the assumption that $f$, although it may now be strongly heterogeneous, converges to some KPP and $L$-periodic function $\tilde{f}$. More precisely, there exists some $C>0$ such that for any $u \geq 0$,
\begin{equation}\label{eqn:eq_CTP}
| f(x,u) - \tilde{f} (x,u) | \leq C e^{-2 \lambda_* x} u,
\end{equation}
where $\lambda_{*}$ is given as above from the periodic equation with reaction term $\tilde{f}$, that is it is the unique solution of $\lambda_*^2 - c^* \lambda_* = \tilde{\mu} (\lambda_*)$, where $\tilde{\mu}$ is defined as $\mu$ with $f$ replaced by $\tilde{f}$, and $c^*$ the minimal speed of pulsating traveling waves of this KPP periodic equation.

We will refer to this situation as ``close-to-periodic", and to the following periodic reaction-diffusion equation
\begin{equation}\label{eqn:eqRD_lim}
\partial_t u = \partial_{xx} u + \tilde{f}(x,u)
\end{equation}
as the limiting problem. Because $\tilde{f}$ is KPP, this problem admits a unique positive stationary solution~$\tilde{p}$ which is also $L$-periodic. It is clear from this uniqueness that $p(x) - \tilde{p} (x) \rightarrow 0$ as~$x \rightarrow +\infty$.

Here, we consider any non trivial initial datum $0 \leq u_0 \leq p$ such that there exists some $D>0$ with $u_0 \equiv 0$ on the half line $[D,+\infty)$. Our goal is to show that the associated solution of the Cauchy problem \eqref{eqn:eqRD} still converges to the traveling wave $U_{c^*}$ with minimal speed of the limiting periodic problem \eqref{eqn:eqRD_lim}. We first begin with a theorem describing the large time behavior of $u$ in the non moving frame:
\begin{thm}\label{th:spread_ctp}
Assume that $f$ is KPP and close-to-periodic, and let $0 \leq u_0 \leq p$ some initial datum such that it is not identically equal to 0. Then the associated solution $u(t,x)$ of the Cauchy problem converges locally uniformly to $p$.
\end{thm}
The theorem above describes the large time behavior of $u$ where $p$ may not be close to $\tilde{p}$, thus where we cannot expect $u$ to be close to $U_{c^*}$, in order to give a more complete picture of the dynamics.\\

We now state our main result in the close-to-periodic case, on the convergence of the profile of propagation:
\begin{thm}\label{th:CV_ctp}
Assume that $f$ is KPP and close-to-periodic, and denote by $U_{c^*}$ the pulsating front, connecting 0 to $\tilde{p}$ with minimal speed~$c^*$, of \eqref{eqn:eqRD_lim}.
Let $0 \leq u_0 \leq p$ some continuous initial datum such that
$$\forall x \geq D, \ u_0 (x)= 0,$$
for some $D>0$.

Then there exists $m(t)=o(t)$ such that the associated solution of the Cauchy problem satisfies, for any $\alpha (t) \rightarrow +\infty$ as $t \rightarrow +\infty$,
$$\|u(t,\cdot)-U_{c^*} (t-m(t),\cdot) \|_{L^\infty (\alpha (t),+\infty)} \rightarrow 0$$
as $t \rightarrow +\infty$.\end{thm}
%

Theorem~\ref{th:CV_p} is partly (recall that it dealt with larger classes of initial data) a corollary of Theorems~\ref{th:spread_ctp} and~\ref{th:CV_ctp} in the particular case $f \equiv \tilde{f}$. One could also consider, following the method presented in this work, roughly periodic environments with a varying period $l(x)$, such as in \cite{GGN} where spreading speeds where studied in slowly varying media. If the period converges quickly enough to some finite value $l_\infty$, one would get convergence to the pulsating traveling wave with minimal speed of the $l_\infty$ periodic problem. If the period converges to $+\infty$, one would get, in some extremal cases, either convergence to some homogeneous traveling waves around each point (when the period grows very quickly) or convergence to the pulsating waves of the problem with large periods (when the period grows very slowly).

However, note that our ``close-to-periodic" assumption is quite strong, as one may expect convergence to the pulsating wave $U_{c^*}$ even if $f$ converges more slowly to $\tilde{f}$. Hence, we expect such results to hold in more general situations, which unfortunately could not be dealt with our method, as it will be made clearer from the proof. Moreover, it is not known how the delay term $m(t)$ behaves in this context, and in particular if and how it is affected by non periodic heterogeneities. This further highlights the difficulty of studying strongly heterogeneous problems rigorously, and gives further motivation in the present work.


\section{Some short preliminaries}\label{sec:twovalues}

We quickly begin by introducing the steepness argument, on which our proofs will strongly rely. This argument directly comes from the zero-number (or intersection-number) argument, which has become a powerful tool in the convergence proofs in semilinear parabolic equations~\cite{An88,Matano78}.\\

Let us first define what we will mean by steepness:
\begin{defi}
Let two functions $u_1 : \R \mapsto \R$ and $u_2 : \R \mapsto \R$. We say that $u_1$ is \textbf{steeper than}~$u_2$ if, for any $x \in \R$ such that $u_1 (x) = u_2 (x)$, then $u_1 (y) \geq u_2 (y)$ for any $y \leq x$ and $u_2 (y) \geq u_1 (y)$ for any $y \geq x$.
\end{defi}
The main idea behind the zero-number argument is that the number of intersections of two solutions of a parabolic equation is decreasing in time, and can therefore serve as a Lyapunov function. We refer to \cite{An88} for details and proofs. As we will aim to compare steepness of some solutions in the sense defined above, we will only consider the particular case of a number of intersections never exceeding one. In the same setting, this powerful tool was previously used by the author in an earlier work~\cite{DGM}, in order to get the convergence to pulsating traveling waves given an Heaviside type initial datum (or, for more complex multistable nonlinearities, to a layer of several fronts that we chose to call a propagating terrace). As mentioned above, our work will rely on this earlier result. Furthermore, the following lemma on the intersection number and steepness properties was shown:

\begin{lem}\label{lem:steep}
Let two solutions $u_1$ and $u_2$ of \eqref{eqn:eqRD} with initial data respectively $u_{0,1}$ piecewise continuous and bounded, and $u_{0,2}$ continuous and bounded. If $u_{0,1}$ is steeper than $u_{0,2}$, then $u_1 (t,\cdot)$ is steeper than $u_2 (t, \cdot)$ for any $t>0$.
\end{lem}
We refer the reader to~\cite{DGM} for the detailed proof. Although this is the only result from the zero-number theory that we will need here, it will be used extensively throughout this work, which is why we state it explicitly here.

\section{Convergence to the pulsating wave in a periodic medium}\label{sec:periodic}

Let us first consider the case of the periodic equation, that is $f \equiv \tilde{f}$ is $L$-periodic, so that~\eqref{eqn:eqRD} and \eqref{eqn:eqRD_lim} are the same. As we mentioned before, our result of convergence of the profile is also the subject of another paper \cite{HNRR}. However, as we consider a larger class of initial data, and as our second main result Theorem~\ref{th:CV_ctp} will rely on the alternative proof we propose here in the periodic case, we include it here. We refer the reader to Lau's paper \cite{Lau85}, from which our method is inspired.

\subsection{Beginning the proof}

For each $k \in \mathbb{N}$, let 
$$t_k=\inf \{t \geq 0 \; | \ u(t_k,kL) = p(0)/2 \}.$$
Our first goal is to show that $u(\cdot +t_k, \cdot + kL)$ converges to $U_{c^*}$. As in \cite{DGM,Lau85}, we look at the steepness of the solution. We separate the proof in two parts and show that any limit of $u(t_k,\cdot)$ is both steeper and less steep than $U_{c^*} (0,\cdot)$, in the sense defined above, which will immediately give the wanted equality.

The fact that it is less steep is in fact a straightforward consequence of a previous paper~\cite{DGM}, where it was proven that $U_{c^*}$ is steeper than any other entire solution of \eqref{eqn:eqRD}. Indeed, both~$U_{c^*}$  and our initial datum $u_0$ are clearly less steep than $H(a-x) p(x)$ for any $a \in \R$, where $H$ denotes the Heaviside function. Let $u(t,x;a (t_k))$ be the solution of the Cauchy problem with such an initial datum, with $a(t_k)$ chosen so that $u(t_k,kL;a(t_k)) = p(0)/2$. It then follows from Lemma~\ref{lem:steep} than
\begin{eqnarray*}
u(t_k,x) & \leq & u(t_k,x;a(t_k)), \ \  x \leq kL,\\
& \geq & u (t_k,x;a(t_k)), \ \ x \geq kL,
\end{eqnarray*}
and
\begin{eqnarray*}
U_{c^*} (0,x-kL) & \leq & u(t_k,x;a(t_k)), \ \  x \leq kL,\\
& \geq & u (t_k,x;a(t_k)), \ \ x \geq kL.
\end{eqnarray*}
Besides, up to extraction of some subsequence, $u(t+t_k,x;a(t_k))$ converges locally uniformly in time and space to an entire solution of \eqref{eqn:eqRD}. As $U_{c^*}$ is steeper than any other entire solution, and from the inequalities above, we get that $u(t_k,x;a(t_k))$ converges to $U_{c^*}$ (even the whole sequence, as it is relatively compact in $C^1_{loc} (\R^2)$) locally uniformly in space. One can check that this convergence is in fact uniform from the asymptotics of $U_{c^*}$ and the fact that $u(t_k,x-L;a(t_k)) \geq u(t_k,x;a(t_k))$ for any $x$ (this inequality follows from the comparison principle, the $L$-periodicity of \eqref{eqn:eqRD} and the fact that it is satisfied for the initial datum $u(0,\cdot;a(t_k))$).

We finally conclude that
\begin{eqnarray*}
u(t_k,x) & \leq & U_{c^*} (0,x-kL) + \zeta (t_k), \ \  x \leq kL,\\
& \geq & U_{c^*} (0,x-kL) - \zeta (t_k), \ \ x \geq kL.
\end{eqnarray*}
with $\zeta (t_k) = \| u(t_k, \cdot;a(t_k)) - U_{c^*} (0,\cdot-kL) \|_\infty \rightarrow 0$ as $t_k \rightarrow +\infty$. Note that $\zeta$ does not depend on the choice of $u_0$ other than through the time $t_k$, a fact which will be used in the close-to-periodic case.

We now want to show some converse inequalities. As mentioned above, this part is inspired by~\cite{Lau85}, which dealt with the homogeneous case.

\subsection{Steeper than the pulsating wave}

By choosing the right shift in time, we assumed that $U_{c^*} (0,0) = p(0)/2$. Moreoever, it is known that all pulsating waves are monotonically increasing in time, so that $\partial_t U_{c^*} (0,0) \geq 2 \gamma$ for some~$\gamma >0$. By uniqueness of the front with speed $c^*$, one can check that the $U_c$ normalized such that $U_c (0,0) = p(0)/2$ converge in $C_{loc}^1$ to $U_{c^*}$. It clearly follows that there exists $\epsilon >0$ small enough, for any $c^* \leq c \leq c^* +\epsilon$ and $-\epsilon \leq t \leq \epsilon$, one has
$$ \partial_t U_c (t,0) \geq \gamma >0.$$
One also has, as $c \rightarrow c^*$, that $\lambda_c$ the smallest solution of $\lambda_c^2 -c\lambda_c =\mu (\lambda_c)$ converges to $\lambda_*$ the unique solution of $\lambda_*^2 -c^* \lambda_* =\mu (\lambda_*)$. More precisely, as $\lambda \mapsto \lambda^2 - \mu (\lambda)$ is known to be analytic and not constant, there exists $K >0$ and $N \in \mathbb{N}$ such that $c- c^* \sim K (\lambda_* - \lambda_c)^N$ as $c \rightarrow c_*$. Furthermore, one has that $N = 2$ from~\cite{Pinsky}.

\begin{rmq} The inequality $N\geq 2$ is more straightforward, and one could actually check that it would be sufficient for the purpose of our proof. Let us briefly prove it for the sake of completeness. Note that
\begin{eqnarray*}
\lambda_c^2 -\lambda_*^2 + c^* \lambda_* -c\lambda_c +\mu (\lambda_*) - \mu (\lambda_c) & = & 0
\end{eqnarray*}
and, as $2\lambda_* - c^* = \mu ' (\lambda_*)$, one gets by a simple computation that:
\begin{eqnarray*}
(\lambda_c -\lambda_*)^2 \left(1 - \frac{\mu'' (\lambda_*)}{2} \right) + \lambda_c (c^* - c)  & = &  o((\lambda_c - \lambda_*)^2).
\end{eqnarray*}
This clearly implies that $N \geq 2$.\end{rmq}

It also follows that the eigenfunction $\phi_{\lambda_c}$ of \eqref{principal-eigen} with $\lambda=\lambda_c$ and normalized such that $\| \phi_{\lambda_c} \|_\infty =1$, converges to the unique eigenfunction $\phi_{\lambda_*}$ of \eqref{principal-eigen}, with $\lambda=\lambda_*$. Up to reducing~$\epsilon$, one can thus assume that for any $c^* \leq c \leq c^* +\epsilon$,
$$\min \phi_{\lambda_c} \geq \frac{\min \phi_{\lambda_*}}{2}  \ \mbox{ and } \ \frac{2\lambda_*}{3} < \lambda_* - \frac{2}{K} (c- c^*)^{1/2} < \lambda_c < \lambda_* - \frac{1}{2K} (c- c^*)^{1/2}.$$
Let now $0<\delta < \min \{\frac{c^*}{20}, \frac{\sqrt{c^*}}{5}, \frac{1}{1000}, \epsilon \}$. We want to show that $u$ is steeper than~$U_c$ the pulsating traveling wave with speed $c=c^*+ \delta$ for any such $\delta$. However, this is clearly not true that $u_0$ is steeper than any traveling wave. In order to use Lemma~\ref{lem:steep}, let us introduce a new initial datum
$$u_{0,r} (x) := \min \left\{p(x), u_0 (x) + \phi_{\lambda_c} (x) e^{-\delta r} e^{-\lambda' x}\right\},$$
where $\lambda' = \lambda_* - \frac{1}{2K} \delta$, so that
$$\left\{
\begin{array}{l}
\displaystyle \lambda' \in \left( \lambda_c, \lambda_* \right),\vspace{3pt}\\
\displaystyle c':= \frac{\lambda'^2 - \mu (\lambda')}{\lambda'} < c^* + \delta^{2}.
\end{array}
\right.
$$
Recall that $\lambda_c < \lambda_* - \frac{1}{2K} \delta^{1/2}$, which insures that $\lambda_c < \lambda'$.

Denote by $u_r$ the associated solution of the Cauchy problem. Similarly as before, let
$$t_k (r)=\inf \{t \geq 0 \; | \ u_r (t_k(r),kL) = p(0)/2 \}.$$
Note that $u_0 \leq u_{0,r} \leq \min \{ p(x) , A e^{-\lambda_* x} + e^{-\lambda' x} \}$ for any $r \geq 0$. As it is already known that the solutions of the periodic reaction-diffusion equation associated with those initial data spread respectively with speed $c^*$ and $c ' < c^* + \delta^{2}$, it follows that $\frac{kL}{c^* +\delta^{2}} \leq t_k (r) \leq t_k \sim \frac{kL}{c^*}$ as $k \rightarrow +\infty$,  uniformly with respect to~$r$. Hence, from the choice of~$\delta$ and for $k$ large enough, we can choose 
$$r_k := 4 \lambda^* \delta t_k,$$
so that
$$\frac{r_k}{5\lambda_* \delta} \leq t_k (r_k) \leq t_k = \frac{r_k}{4\lambda_* \delta}.$$
It is clear that $r_k \rightarrow +\infty$ as $k \rightarrow +\infty$, and $u_{0,r} \rightarrow u_0$ locally uniformly as $ r\rightarrow +\infty$. Let us now show the following lemma, which states that $u$ and $u_{r_k}$ are close from each other:
\begin{lem}
There exists $\eta_{1,\delta} (r)\rightarrow 0$ as $r \rightarrow +\infty$, which does not depend on $u_0$, and such that for any $k$ large enough and $x \in [(c^*-\delta^{2})t_k (r_k),+\infty)$, one has
$$| u(t_k (r_k),x)- u_{r_k} (t_k (r_k),x)| \leq \eta_1 (r_k).$$
\end{lem}
\begin{proof}
The function $v(t,x)=u_{r_k} (t,x)-u(t,x) \geq 0$ satisfies
$$\partial_t v = \partial_x^2 v + f(x,u_{r_k}) -f(x,u).$$
Furthermore, the KPP hypothesis implies that for any $x$,
$$ f(x,u_{r_k})-f(x,u) \leq \partial_u f(x,0) (u_{r_k}-u).$$
Therefore, $v$ satisfies
$$\partial_t v - \partial_x^2 v - \partial_u f(x,0) v \leq 0.$$
Besides, one can check that
$$\overline{v} (t,x) := \frac{ \phi_{\lambda_*} (x)}{\min \phi_{\lambda_*}} e^{-\delta r_k} e^{-\lambda' (x- c' t)}$$
satisfies
$$\partial_t \overline{v} - \partial_x^2 \overline{v} - \partial_u f(x,0) \overline{v} = 0.$$
Since $v(0,x) \leq \overline{v} (0,x)$, one gets from the parabolic comparison principle that for any $x \geq (c^* -\delta^{2})t \geq (c'-2\delta^{2})t$:
\begin{eqnarray*}
0 \leq u_{r_k} (t_k (r_k),x) - u(t_k(r_k),x) & \leq & \frac{ \phi_{\lambda_*} (x)}{\min \phi_{\lambda_*}}e^{-\delta r_k} e^{-\lambda' (-2\delta^{2} t_k (r_k))}\\
& \leq & \frac{1}{\min \phi_{\lambda_*}} e^{-\delta r_k + 2\lambda' \delta^{2} t_k(r_k)}\\
& \leq & \frac{1}{\min \phi_{\lambda_*}} e^{-\frac{\delta}{2} r_k} =: \eta_{1,\delta} (r_k).
\end{eqnarray*}
This concludes the proof of the lemma.
\end{proof}\\

We now only need to ``cut" $U_c$ to make it less steep than $u_{0,r}$. Let $\overline{r}_k $ such that
\begin{equation}\label{eqn:Uc_cut}
U_c (t,x +ct) \leq \frac{\min \phi_{\lambda_*}}{2} e^{-\delta r_k},
\end{equation}
for any $t \in \mathbb{R}$ and $x \geq \overline{r}_k$. From the asymptotics of $U_c$, there exists $D'$ such that
$$U_c (t,x+ct) \leq 2 B(c)  e^{-\lambda_c x},$$
for any $t \in \mathbb{R}$ and $x \geq D'$. Therefore, \eqref{eqn:Uc_cut} holds if
$$\overline{r}_k \geq D' \; \mbox{ and } \; 2 B(c)  e^{-\lambda_c \overline{r}_k} \leq \frac{\min \phi_{\lambda_*}}{2} e^{-\delta r_k},$$
that is,
$$\overline{r}_k \geq \max \left\{ D' \; , \ \frac{\ln (4B(c)/ \min \phi_{\lambda_*})}{\lambda_c} + \frac{\delta r_k}{\lambda_c} \right\}.$$
For instance, as $r_k \rightarrow +\infty$ as $k \rightarrow + \infty$, one can choose
$$\overline{r}_k = 2\frac{\delta r_k}{\lambda_c}.$$
Let now $\tilde{U}_c$ be defined as
$$\tilde{U}_c (t,x)= 
\left\{
\begin{array}{l}
U_c (t,x) \mbox{ if } x \geq \overline{r}_k + ct, \\
0 \mbox{ otherwise.}
\end{array}
\right.$$
Choose $s_k$ such that $U_c (s_k+t_k (r_k),kL) = u_{r_k} (t_k (r_k),kL) = 1/2$. We now claim the following lemma:
\begin{lem}
For $k$ large enough (depending on $u_0$ only through $A$), the function $\tilde{U}_c (s_k,\cdot)$ is less steep than $u_{0,r_k}$. 
\end{lem}
\begin{proof}
Let us first estimate $s_k$. Since~$U_c$ moves with speed $c=c^* +\delta$, one has that
$$\frac{kL}{s_k+t_k(r_k)} = c^*+\delta$$
for each $k\in \mathbb{N}$. We have also already seen that $c^* \leq \liminf \frac{kL}{t_k (r_k)} \leq \limsup \frac{kL}{t_k (r_k)} \leq c^* + \delta^2$ as~$k \rightarrow +\infty$. It follows that
$$ \frac{c^*}{c^* + \delta} \leq \liminf \frac{s_k + t_k (r_k)}{t_k (r_k)}  \leq \limsup \frac{s_k + t_k (r_k)}{t_k (r_k)} \leq \frac{c^*  + \delta^2}{c^* + \delta}$$
as $k \rightarrow +\infty$. Hence,
$$\limsup_{k \rightarrow +\infty} \left| \frac{s_k}{t_k (r_k)} + \frac{\delta}{c^*} \right| \leq \frac{\delta^2}{c^* +\delta} + \frac{\delta^2}{c^* (c^* +\delta)} < \frac{1}{10} \; \frac{\delta}{c^*}.$$
Hence, from the choice of $\delta$ and $r_k$, we get for $k$ large enough that
$$s_k \leq - \frac{ r_k}{6\; c^* \lambda_* }.$$
We can now check that $\tilde{U}_c (s_k,x)$ is less steep than $u_{0,r_k}$. First, it is clear from its construction that~$\tilde{U}_c (s_k,0)$ lies below $u_{0,r_k} (0)$. On the other hand, as $U_c$ has a slower decay than $u_{0,r_k}$ as~$x \rightarrow +\infty$, the front $U_c$ lies above $u_{0,r_k}$ on the far right of the domain. Thus, there exists some $y_k \geq 0$ such that $\tilde{U}_c (s_k,y_k) = u_{0,r_k} (y_k)$. Since $u_{0,r_k}>0$ and $\tilde{U}_c < p$, this means that
$$U_c (s_k,y_k)= u_0 (y_k) + \phi_{\lambda_c} (y_k) e^{-\delta r_k} e^{-\lambda'  y_k}.$$
Since $cs_k \rightarrow -\infty$ as $k \rightarrow +\infty$, one has 
\begin{eqnarray*}
U_c (s_k,x) & \leq  & 2 B(c) e^{-\lambda_c (x - cs_k)}
\end{eqnarray*} 
for any $x \geq 0$ and $k$ large enough. Therefore,
$$ u_0 (y_k) + \phi_{\lambda_c} (y_k) e^{-\delta r_k} e^{-\lambda 'y_k} \leq 2 B(c) e^{-\lambda_c (y_k-cs_k)}.$$
Hence $$\frac{\min \phi_{\lambda_*}}{2} e^{-\delta r_k} e^{(\lambda_c -\lambda') y_k} \leq 2 B(c) e^{c\lambda_c s_k}$$
and $$y_k \geq \frac{1}{\lambda' - \lambda_c} \left( \ln (4 B(c) / \min \phi_{\lambda_*}) - \delta r_k - \lambda_c c s_k \right).$$
We get for $k$ large enough, and since 
$$\frac{c\lambda_c}{6c^* \lambda_*} - \delta \geq \frac{1}{9} - \delta \geq \frac{1}{10},$$
that
$$y_k \geq \frac{r_k}{10 (\lambda' -\lambda_c)} \rightarrow +\infty.$$
It also follows from the asymptotics of $U_c$ that
$$\frac{ \partial_x \tilde{U}_c (s_k,y_k)}{\tilde{U}_c (s_k,y_k) } =\frac{ \partial_x U_c (s_k,y_k)}{U_c (s_k,y_k) } = -\lambda_c + \frac{\phi '_{\lambda_c} (y_k)}{\phi_{\lambda_c} (y_k)}.$$
Besides, 
\begin{eqnarray*}
\partial_x u_{0,r_k} (y_k) & = & u'_0 (y_k) + \left(\phi '_{\lambda_c} (y_k) -\lambda' \phi_{\lambda_c} (y_k)\right) e^{-\delta r_k} e^{-\lambda' y_k}\\
& \leq & A e^{-\lambda_* y_k} + \left(\phi '_{\lambda_c} (y_k) -\lambda' \phi_{\lambda_c} (y_k)\right) e^{-\delta r_k} e^{-\lambda' y_k}\\
& \leq & \left(\frac{2A}{\min \phi_{\lambda_*}} e^{\delta r_k} e^{(\lambda'-\lambda_*) y_k}  -  \lambda' +  \frac{\phi '_{\lambda_c} (y_k)}{\phi_{\lambda_c} (y_k)}  \right)  \phi_{\lambda_c} (y_k) e^{-\delta r_k} e^{-\lambda' y_k}\\
& \leq & \left(\frac{2A}{\min \phi_{\lambda_*}} e^{\delta r_k} e^{(\lambda'-\lambda_*) y_k}  -  \lambda' +  \frac{\phi '_{\lambda_c} (y_k)}{\phi_{\lambda_c} (y_k)}  \right)  u_{0,r_k} (y_k).
\end{eqnarray*}
Thus, since $u_{0,r_k} (y_k)=\tilde{U}_c (s_k,y_k)$:
\begin{eqnarray*}
\frac{ \partial_x \tilde{U}_c (s_k,y_k) -\partial_x u_{0,r_k} (y_k)}{\tilde{U}_c (s_k,y_k)}&  \geq  &\lambda' -\lambda_c   -\frac{2A}{\min \phi_{\lambda_*}} e^{\delta r_k} e^{(\lambda ' - \lambda_*) y_k}\\
& \geq & \lambda' -\lambda_c   -\frac{2A}{\min \phi_{\lambda_*}} e^{\delta r_k} e^{(\lambda ' - \lambda_*)  \frac{r_k}{10 (\lambda' -\lambda_c)} }\\
\end{eqnarray*}
Recall that $\lambda' -\lambda_* = -\frac{1}{2K} \delta$ and $\lambda_* - \lambda_c \leq \frac{2}{K} \delta^{1/2}$, hence
\begin{eqnarray*}
\frac{2A}{\min \phi_{\lambda_*}} e^{\delta r_k} e^{(\lambda ' - \lambda_*)  \frac{r_k}{10 (\lambda' -\lambda_c)} } & \leq & \frac{2A}{\min \phi_{\lambda_*}} e^{\delta r_k} e^{-\frac{\delta^{1/2}}{4-\delta^{1/2}} \frac{r_k}{10} } \\
& \leq & \frac{2A}{\min \phi_{\lambda_*}} e^{r_k (\delta -\frac{\delta^{1/2}}{30}) } \\
& \rightarrow & 0
\end{eqnarray*} 
as $k \rightarrow +\infty$ provided that $\delta^{1/2} < \frac{1}{30}$, which is true from our choice of $\delta$. We conclude that $\partial_x \tilde{U}_c (s_k,y_k) -\partial_x u_{0,r_k} (y_k) >0$, for any point $y_k$ where $\tilde{U}_c (s_k,\cdot)$ intersects $u_{0,r_k} (\cdot)$. Namely, $u_{0,r_k} (\cdot)$ is steeper than $\tilde{U}_c (s_k,\cdot)$.\end{proof}\\

From Lemma~\ref{lem:steep}, one immediately gets that $u_{r_k} (t_k (r_k),\cdot)$ is steeper than $\tilde{u} (t_k (r_k),\cdot)$, where $\tilde{u}$ is the solution of the Cauchy problem with initial datum $\tilde{U}_c (s,x)$. It remains to prove that $\tilde{u} (t_k (r_k),\cdot) $ is close to $U_c (s + t_k (r_k),\cdot)$ to be able to conclude. This is the purpose of the next lemma:
\begin{lem}
There exists $\eta_{2,\delta} (r) \rightarrow 0$ as $r \rightarrow +\infty$, which does not depend on $u_0$, such that for any~$k$ large enough and $x > (c^*-\delta^2)  t_k (r_k)$, one has
$$0 \leq U_c (s_k + t_k (r_k),x) - \tilde{u} (t_k (r_k),x) \leq \eta_{2,\delta} (r_k).$$
\end{lem}


\begin{proof}
As before, the function $w (t,x)=U_c (s_k + t,x)-\tilde{u} (t,x)$ is a subsolution of 
$$\partial_t w = \partial_x^2 w + \partial_u f(x,0)w.$$
On the other hand, for any positive constant $C$, the above equation admits the following solution:
$$\overline{w} (t,x) := C \phi_{\lambda_*} e^{-\lambda_* (x-c^*t)}.$$
Hence, by the comparison principle, $w \leq \overline{w}$ provided that $w(t=0,\cdot) \leq \overline{w} (t=0,\cdot)$. This is true for
$$C =\frac{e^{\lambda_* (\overline{r}_k + c s_k)}}{\min \phi_{\lambda_*}}.$$
Indeed, one then has $\overline{w} (t=0,x) \geq 1 \geq w(t=0,x)$ for any $x \leq \overline{r}_k + cs_k$, and $\overline{w} (t=0,x) \geq w(t=0, x)= 0$ for any $x \geq \overline{r}_k + cs_k$.

It follows that for any $x > (c^* -\delta^2) t_k (r_k)$,
\begin{eqnarray*}
0 \leq w (t_k (r_k),x) & \leq & \frac{1}{\min \phi_{\lambda_*}} \times e^{\lambda_* (\overline{r}_k + c s_k)} \times e^{-\lambda_* (x-c^*t_k (r_k))}\\
& \leq &  \frac{1}{\min \phi_{\lambda_*}} \times e^{\lambda_* (\overline{r}_k + c s_k)} \times e^{\lambda_* \delta^2 t_k (r_k)}\\
& \leq & \displaystyle  \frac{1}{\min \phi_{\lambda_*}}\times  e^{\frac{2 \lambda_* \delta r_k}{\lambda_c} + \lambda_* c s_k + \frac{ \delta r_k}{2}}.
\end{eqnarray*}  
Recall that $s_k \leq - \frac{r_k}{6\; c^* \lambda_* }.$ We conclude, again from the choice of $\delta$ small enough, that for any $x > (c^* - \delta^2) t_k (r_k)$:
\begin{eqnarray*}
0 \leq w (t_k (r_k),x) & \leq & \displaystyle e^{\frac{2 \lambda_* \delta r_k}{\lambda_c} - \frac{c r_k}{6c^*} + \frac{ \delta r_k}{2}}\\
& \leq & \displaystyle e^{3 \delta r_k - \frac{c r_k}{6c^*} + \frac{ \delta r_k}{2}}\\
& \leq & e^{-\frac{r_k}{7}}
\end{eqnarray*} 
which ends the proof by letting $\eta_2$ this last expression.
\end{proof}\\

From all the above, and by noting that $kL \geq (c^* - \delta^2 /2 )t_k (r_k)$ for $k$ large enough, we get that for any small $\delta$ and large $k$,
\begin{eqnarray*}
u(t_k (r_k),x) & \leq & U_{c^*}(0,x-kL) + \eta (r_k) + \epsilon_1 (\delta), \ \ x \geq kL,\\
& \geq & U_{c^*} (0,x-kL) - \eta (r_k) - \epsilon_1 (\delta), \ \  kL - \frac{\delta^2}{2} t_k (r_k)  \leq x \leq kL.
\end{eqnarray*}
where $\eta = \eta_{1,\delta} + \eta_{2_\delta}$ and $\epsilon_1 (\delta) \rightarrow 0$ as $\delta \rightarrow 0$ comes from the uniform convergence of $U_{c^* +\delta} (0,\cdot)$ to~$U_{c^*} (0,\cdot)$. We also have an opposite inequality, as in the previous subsection:
\begin{eqnarray*}
u(t_k (r_k),x) & \geq & U_{c^*} (0,x-kL) - \zeta (t_k (r_k)) , \ \ x \geq kL,\\
& \leq & U_{c^*} (0,x-kL) + \zeta (t_k (r_k)), \ \   x \leq kL,
\end{eqnarray*}
where $\zeta (t) \rightarrow 0$ as $t \rightarrow +\infty$.\\

We aim to conclude by passing to the limit as $\delta \rightarrow 0$ and $k \rightarrow +\infty$. To do so, one needs to proceed carefully as $r_k$ and $\eta$ actually also depend on $\delta$.

Let some sequence $\delta_n \rightarrow 0$. From our first lemma, one can find $k_n$, such that for any $k \geq k_n$, we have $\eta (r_k) \leq 1/n$ and $\zeta (t_k (r_k)) \leq 1/n$. We can easily assume without loss of generality that~$k_n \rightarrow +\infty$ as $n \rightarrow +\infty$. Before we can conclude, it remains to check that 
$$t_k- t_k (r_k) \rightarrow 0 \ \mbox{ for } \ k \geq k_n \ \mbox{ as } \ n \rightarrow +\infty .$$
Recall first that $t_k (r_k) \leq t_k$. Moreover, $\eta_1 (r_k) \leq \eta (r_k) \leq 1/n$, hence $u (t_k (r_k),kL) > 1/2 - 1/n$. From standard parabolic estimates, $\partial_x u$ is globally bounded and we can put a non trivial compactly supported function $\underline{u}_0 (x)$ below $u (t_k (r_k),x+kL)$ for any $k$, such that the associated solution $\underline{u} (t,x)$ of the Cauchy problem converges to $p$ locally uniformly in time (this is true for any non trivial compactly supported function \cite{BHR-I}). In particular, there exists $T$ such that $\underline{u} (T,kL) > \frac{p(0)}{2}$. It follows that for any $k \geq k_n$, one now has that $0 \leq t_k - t_k (r_k)  \leq T$. 

We then show, as announced, that $t_k - t_k (r_k) $ converges to 0 as $\delta_n \rightarrow 0$ uniformly with respect to $k \geq k_n$. Let any subsequence $n_j$ and $k_j \geq k_{n_j}$ such that $t_{k_{n_j}} - t_{k_{n_j}} (r_{k_{n_j}}) $ converges to some constant $T'$. Up to extraction of a subsequence and thanks to parabolic estimates, one can assume that the sequence $u (\cdot + t_{k_{n_j}} (r_{k_{n_j}}),\cdot + k_{n_j}L)$ converges in $C^1_{loc} (\mathbb{R})$ to an entire solution $u_\infty$ of~\eqref{eqn:eqRD}. By passing to the limit in the inequalities above, one gets that
\begin{eqnarray*}
u_\infty (T',x) & = & U_{c^*} (0,x), \ \ x \geq 0 .
\end{eqnarray*}
It follows that $u_\infty (\cdot  , \cdot) \equiv U_{c^*} (\cdot - T' , \cdot)$ is a shift of the front $U_{c^*}$. Hence, from the monotonicity of~$U_{c^*}$, we have that $\partial_t u (t+ t_{k_{n_j}} (r_{k_{n_j}}),k_{n_j} L) \geq \gamma /2 >0$ for large $j$ and $|t| \leq \epsilon$. But we know that $u (t_{k_{n_j}} (r_{k_{n_j}}),k_{n_j} L) > p(0)/2 - 1/{n_j}$, thus $$u(t_{k_{n_j}} (r_{k_{n_j}})+ 2/(\gamma n_j),k_{n_j} L) \geq p(0)/2.$$
It immediately follows that $t_k - t_k (r_k) \leq \frac{2}{\gamma n_j}$, hence $T' =0$. We can now even conclude that~$t_k - t_k (r_k)$ converges to 0 uniformly with respect to $k \geq k_n$.

Let us now check that the sequence $u (t_k ,\cdot)$ converges uniformly in the right half-space to~$U_{c^*} (0,\cdot)$. Let any $\varepsilon >0$. Recall that 
\begin{eqnarray*}
u(t_k ,x) & \leq & U_{c^*}(0,x-kL) + \frac{1}{n} + \epsilon_1 (\delta_n) + \epsilon_2 (n) , \ \ x \geq kL,\\
& \geq & U_{c^*} (0,x-kL) - \frac{1}{n}  - \epsilon_1 (\delta_n)- \epsilon_2 (n), \ \  kL - \frac{\delta_n^2}{2} t_k  \leq x \leq kL,
\end{eqnarray*}
and \begin{eqnarray*}
u(t_k ,x) & \geq & U_{c^*} (0,x-kL) - \frac{1}{n}  , \ \ x \geq kL,\\
& \leq & U_{c^*} (0,x-kL) + \frac{1}{n} , \ \   x \leq kL,
\end{eqnarray*}
where $k \geq k_n$ and $\epsilon_2 (n) \rightarrow 0$ as $n \rightarrow +\infty$ comes from the uniform convergence of $u (t_k (r_k),\cdot)$ to~$u(t_k,\cdot)$ (since $t_k - t_k (r_k) \rightarrow 0$). We choose $n$ large enough so that $$\frac{1}{n} + \epsilon_1 (\delta_n) + \epsilon_2 (n) \leq \varepsilon.$$
Up to increasing $k_n$, we can assume from the fact that $u$ spreads with speed $c^*$, that
$$u(t_k,x) \geq 1 - \varepsilon, \ \ 0 \leq x \leq kL - \frac{\delta_n^2}{2}t_k,$$
$$U_{c^*} (0,x) \geq 1 -\varepsilon, \ \ x \leq - \frac{\delta_n^2}{2}t_k,$$
for any $k \geq k_n$. From all the inequalities above, one can then conclude that for any $\varepsilon >0$, there exists $k_n$ such that
$$\sup_{k \geq k_n} \| u(t_k,\cdot ) - U_{c^*} (0,\cdot - kL) \|_{L^\infty (\R^+)} \leq \varepsilon.$$
We conclude that $u (t_k ,\cdot+kL)$ converges uniformly in the half-space $\R^+$ to $U_{c^*} (0,\cdot)$. From parabolic estimates, this convergence is also locally uniform in time.
\begin{rmq}
When $\liminf u_0 >0$ as $x \rightarrow -\infty$, it is known that the associated solution $u$ of the Cauchy problem converges uniformly to 1 in the half-space $x \leq ct$ for any $c < c^*$. Hence, by proceeding exactly as above, we would get the convergence to $U_{c^*}$ uniformly in the whole space $\mathbb{R}$.
\end{rmq}

\subsection{Ending the proof}\label{sec:ending}

Finally, to get the wanted convergence result as $t\rightarrow +\infty$, let $j (t) \in \mathbb{N}$ such that
$$j(t) \frac{L}{c^*} \leq t < (j(t) +1)\frac{L}{c^*},$$
and $m(t)$ the piecewized affine function defined by
$$m(t)= t_{j(t)} -t \ \mbox{ if } \ t = j(t) \frac{L}{c^*}.$$
Note that from the previous section and the convergence to the pulsating traveling wave minimal speed around the times $t_k$, it easily follows that $t_{k+1} - t_k \rightarrow \frac{L}{c^*}$ as $k \rightarrow \infty$. Recall also that $t_{j(t)} \sim \frac{j(t) L}{c^*}$. Therefore, $m(t)=o(t)$ and $t+m(t)-t_{j(t)} \sim t - j(t) \frac{L}{c^*}$. It follows from the convergence of $u (t_{j(t)},\cdot)$ to $U_{c^*} (0,\cdot -kL)$ uniformly in $\R^+$ and locally in time that
$$u (t+ m(t) ,x) - U_{c^*} \left(t - j(t) \frac{L}{c^*},x - j(t)L \right) \rightarrow 0,$$
thus
$$u (t+ m(t) ,x) - U_{c^*} \left(t ,x \right) \rightarrow 0,$$
where the convergence holds uniformly in $\R^+$ as $t \rightarrow +\infty$. This concludes the proof of the convergence to the pulsating wave with minimal speed in a purely periodic environment.

\section{Spreading in a close-to-periodic medium}

In this section and the next, as $f$ may no longer be periodic and $\tilde{f}$ is the limiting periodic reaction term, $\tilde{\mu} (\lambda)$ will now denote the following principal eigenvalue:
\begin{equation}\label{principal-eigen_tilde}
\left\{
\begin{array}{rcl}
\displaystyle -\partial_{xx} \phi_\lambda + 2 \lambda \partial_x \phi_\lambda - \frac{\partial \tilde{f}}{\partial u} (x,0 ) \phi_\lambda & = & \tilde{\mu} (\lambda) \phi_\lambda \ \mbox{ in } \R, \vspace{3pt}\\
\phi_\lambda >0 \mbox{ and  $L$-periodic}.
\end{array}
\right.
\end{equation}
We recall that $\lambda \mapsto \lambda^2 - \tilde{\mu} (\lambda)$ is an analytic and non constant function of $\lambda$, and that we assumed that $\tilde{\mu} (0) <0$ ($\tilde{f}$ is KPP), so that there exists a unique solution $\lambda_*$ of $\lambda^2 - c^* \lambda = \tilde{\mu} (\lambda)$ with~$c^* =\inf_{\lambda \in \R} \frac{\lambda^2 -\tilde{\mu} (\lambda)}{\lambda}$, and two solutions of $\lambda^2 - c \lambda = \tilde{\mu} (\lambda)$ for any $c > c^*$.\\

We begin our study of the close-to-periodic problem by looking at spreading speeds of solutions. We first show that the solution $u (t,x)$ of the Cauchy problem with equation~\eqref{eqn:eqRD} and an initial datum satisfying the assumptions of Theorem~\ref{th:CV_ctp} spreads with speed~$c^*$ in the following sense:
\begin{lem}\label{lem:spread}
Let $u(t,x)$ the solution of \eqref{eqn:eqRD} associated with an initial datum $0 \leq u_0 \leq p$ such that
$$\forall x \geq D, \ u_0 (x)= 0,$$
for some $D>0$.
Then $u(t,x)$ converges locally uniformly to $p$ and spreads to the right with speed~$c^*$:
$$\forall c < c^* \ , \ \lim_{t \rightarrow +\infty} \sup_{0 \leq x \leq ct } |u(t,x) - p(x)| =0,$$
$$\forall c > c^* \ , \ \lim_{t \rightarrow +\infty} \sup_{x \geq ct } |u(t,x) | =0.$$
\end{lem}
\begin{rmq}
Theorem~\ref{th:spread_ctp} is an immediate corollary of this lemma. Indeed, for any non trivial initial datum $u_0$ between 0 and $p$, we can put below it another initial datum satisfying the assumptions of this lemma. We clearly get, using the comparison principle, the locally uniform convergence to $p$ (and even in the moving frames with speed smaller than $c^*$ to the right).
\end{rmq}
\begin{proof}For any $y \in \R$ large enough, the function $\tilde{f} (x,u) + C e^{-2\lambda_* y} u$ is periodic and KPP. We can denote by $c_y^*$ the minimal speed of pulsating traveling wave solutions of the periodic reaction-diffusion equation with reaction term $\tilde{f} (x,u) + C e^{-2\lambda_* y} u$. It is clear that it can be computed as
$$c^*_y = \inf_{\lambda \in \R} \frac{\lambda^2 - \tilde{\mu} (\lambda) + C e^{-2\lambda_* y}}{\lambda},$$
and thus $c^*_y \geq c^*$ and $c^*_y \rightarrow c^*$ as $y \rightarrow +\infty$. The main idea of the proof is that whatever the spreading speed is before the profile reaches some large~$y$, the spreading speed will then become very close to~$c^*$ from this point.

Let us first fix $c>c^*$, and $c^* < c' < c$. Since $c^*_y$ converges to $c^*$ as $y$ goes to $+\infty$, one can define, for $y$ large enough, $\lambda_{c',y}$ the smallest solution of $\lambda^2 - c' \lambda - \tilde{\mu} (\lambda) + Ce^{-2\lambda_* y} =0$. It immediately follows that $\lambda_{c',y} \rightarrow \lambda_{c'}$ as $y \rightarrow +\infty$, with $\lambda_{c'}$ given by the smallest solution of $\lambda^2 - c' \lambda - \tilde{\mu} (\lambda) =0$. Besides, it is easy to see that $\lambda_{c'} < \lambda_*$.
%

One can now check from the KPP hypothesis that for any $C_1 >0$,
$$\overline{u}_1 (t,x) := \min \{ p(x), \; C_1 e^{-\sqrt{M} (x - 2\sqrt{M} t)} \},$$
where $$M:= \max \{ \sup_{x\in \R} \partial_u f (x,0) , \; \lambda_*^2 \},$$
is a supersolution of \eqref{eqn:eqRD}. Moreover, one can choose $C_1$ such that $u_0 (x) \leq \overline{u}_1 (0,x)$ for any $x \in \R$. Hence, $u(s,x) \leq \overline{u}_1 (s,x)$ for any $t>0$ and $x \in \R$. Let now $t_y$ be the smallest time such that $\overline{u}_1 (t_y,y)=1$. Up to increasing $y$, one can assume without loss of generality that $t_y >0$. In particular, this means that $$\overline{u}_1 (t_y,y) = C_1 e^{-\sqrt{M} (y - 2 \sqrt{M} t_y)} =1.$$
Hence, 
$$\overline{u}_1 (t_y, x) = 
\left\{
\begin{array}{l}
p(x) \mbox{ if } x \leq y,\\
p(y) e^{-\sqrt{M} (x - y)} \mbox{ if } x \geq y.
\end{array}
\right.
$$
Let now
$$\overline{u}_2 (t,x) := \min \left\{ p(x), \frac{\phi_{\lambda_{c',y}} (x)\; p(y)}{\min \phi_{\lambda_{c',y}}}  e^{-\lambda_{c',y} (x -y-c' t)} \right\}.$$
From the choice of $M$ and $t_y$, one has that $u (t_y ,x) \leq \overline{u}_1 (t_y,x) \leq \overline{u}_2 (0,x)$ for any $x \in \R$. Besides, the function $\overline{u}_2 (t,x)$ is a supersolution of \eqref{eqn:eqRD}. Indeed, on one hand, it is a supersolution of the reaction-diffusion equation with reaction term $\tilde{f} + Ce^{-2\lambda_* y}u$, and on the other hand, for any $(t,x)$ such that $\overline{u}_2 (t,x) < 1$, then $x >y+c' t \geq y$, hence $f (x,\overline{u}_2 (t,x)) \leq \tilde{f} (x,\overline{u}_2 (t,x)) + C e^{-2\lambda_* y}$. 

Therefore, for any $t > t_y$ and $x\in \R$, one has from the parabolic comparison principle that
$$u (t,x) \leq \overline{u}_2 (t-t_y,x).$$
In particular, as $t \rightarrow +\infty$,
$$u (t,x+ct) \leq p(y) e^{-\lambda_{c',y} (x+(c-c') t-y)} \rightarrow 0$$
uniformly with respect to $x  \geq 0$. This concludes the first part of the proof of our lemma, namely that $u$ spreads with speed less than $c^*$.\\

In order to construct a subsolution and prove that the spreading speed is in fact exactly~$c^*$, we use the method in Section~4 of \cite{BHNadin}. We introduce the following principal eigenvalue problem, for any $0 \leq c < c^*$:
\begin{equation}\label{another_eigen}
\left\{
\begin{array}{l}
\partial_t \phi_{c,R,y} - \partial_{xx} \phi_{c,R,y} - c \partial_x \phi_{c,R,y} - \partial_u \tilde{f} (x+y+ct,0)  \phi_{c,R,y} = \tilde{\mu}_{c,R} \phi_{c,R,y} , \vspace{3pt} \\
\phi_{c,R,y} (t,x) >0 \mbox{ in } \R \times (-R,R) ,\vspace{3pt} \\
\phi_{c,R,y} \mbox{ is $\frac{L}{c}$-periodic,}\vspace{3pt} \\
\phi_{c,R,y} =0 \mbox{ on } \R \times \{-R,R\}.
\end{array}
\right.
\end{equation}
It has been proved in \cite{Nadin} that $\tilde{\mu}_{c,R} \rightarrow \max_{\lambda \in \R} ( \tilde{\mu} (\lambda) -\lambda^2 +c\lambda)$ as $R\rightarrow +\infty$. Since $c< c^*$, one can find $R$ large enough such that $\tilde{\mu}_{c,R} <0$. Extend $\phi_{c,R,y}$ by 0 outside $(-R,R)$, and let $\psi_{c,R,y} (t,x)=\phi_{c,R,y} (t,x-y-ct)$. 

One can check that for any $\kappa$ small enough and $y$ large enough, the function $\kappa \psi_{c,R,y}$ is a subsolution of \eqref{eqn:eqRD}:
\begin{eqnarray*}
&& \partial_t \kappa \psi_{c,R,y}  - \partial_{xx}  \kappa \psi_{c,R,y} - f(x,\kappa \psi_{c,R,y})\\
  & \leq & \partial_t \kappa \psi_{c,R,y}  - \partial_{xx}  \kappa \psi_{c,R,y} - \tilde{f} (x,\kappa \psi_{c,R,y}) + Ce^{-2\lambda_*(y-R)} \kappa \psi_{c,R,y}\\
  & \leq & \partial_t \kappa \psi_{c,R,y}  - \partial_{xx}  \kappa \psi_{c,R,y} - \partial_u \tilde{f} (x,0)\kappa \psi_{c,R,y} + o (\kappa \psi_{c,R,y}) + Ce^{-2\lambda_* (y-R)} \kappa \psi_{c,R,y}\\
& \leq & \kappa \psi_{c,R,y} \times (\tilde{\mu}_{c,R} + o(\kappa) + Ce^{-2\lambda_* (y-R)}) <0.
\end{eqnarray*}
Note that we made use of the fact that the support of $\psi_{c,R,y} (t,\cdot)$ is included in the set $\{ x \geq y -R \}$ for any $t \geq 0$.

Choose first $c = 0$. It is clear that $\phi_{c,R,y}$ actually does not depend on time by uniqueness of the principal eigenfunction up to multiplication by a factor. By the parabolic comparison principle, one then gets that the solution of \eqref{eqn:eqRD} with initial datum $\kappa \psi_{0,R,y}$ is increasing in time. As it is also bounded from above by $p$ (up to decreasing $\kappa$), and positive for any $t>0$ by the strong comparison principle, we get that it converges locally uniformly to a stationary solution $0<q \leq p$ of~\eqref{eqn:eqRD}. Let us check that $\liminf_{x \rightarrow +\infty} q (x) >0$, proceeding by contradiction. If there exists some sequence~$y_n \rightarrow +\infty$ such that $q (y_n) \rightarrow 0$, one can find $\kappa_n \rightarrow 0$ such that $\kappa_n \psi_{0,R,y_n} \leq q$ with equality on some point. But for $n$ large enough, we know that $\kappa_n \psi_{0,R,y_n} $ is a subsolution for \eqref{eqn:eqRD}, so that~$q \equiv\kappa_n \psi_{0,R,y_n}$ according to the strong maximum principle. This is a clear contradiciton with the fact that $q$ is positive. We conclude that $\liminf_{x \rightarrow +\infty} q (x) >0$ and hence, by uniqueness, we get that $q \equiv p$. Finally, as it is clear that $u(1,\cdot)>0$ from the strong maximum principle, one can choose $\kappa$ small enough so that $p \geq u(1,\cdot) \geq \kappa \psi_{0,R,y}$, so that $u$ also converges locally uniformly to $p$.

Now choose any $0 < c < c^*$ and let us prove that $u$ converges to $p$ as $t \rightarrow +\infty$ on the domain $0 \leq x \leq ct$. Proceed by contradiction and assume that there exists some sequence $(t_n,x_n)$ such that $t_n \rightarrow +\infty$, $0 \leq x_n \leq c t_n$ and $\inf_{n \in \mathbb{N}} | u(t_n,x_n) - p(x_n) |>0$. Note that, since we already know that $u$ converges locally uniformly to $p$, we have that $x_n \rightarrow +\infty$. Moreover, up to extraction of some subsequence, there exists $0\leq c'<c^*$ such that
$$x_n = c' t_n + o(t_n).$$
Let now $c'' \in (c',c^*)$, so that $t_n - \frac{x_n}{c''} \rightarrow +\infty$. As above, we can assume that $u(1,\cdot) \geq \kappa \psi_{c'',R,y}$. From the comparison principle, we get that for any $t >1$, one has that $u(t,\cdot) \geq \kappa \psi_{c'',R,y} (t,\cdot)$. In particular, for any $x$,
$$u \left( \frac{x_n - y}{c''},x+x_n \right) \geq \kappa \phi_{c'',R,y} \left(\frac{x_n -y }{c''},x \right).$$
Thus, 
$$u(t_n,x_n) \geq \underline{u}_1 \left( t_n - \frac{x_n-y}{c''} , 0; x_n\right)$$
where $\underline{u}_1 (\cdot,\cdot;x_n)$ is the solution of
\begin{equation}\label{eqn:plop}
\partial_t \underline{u}_1 = \partial_{xx} \underline{u}_1 + \tilde{f}(x+x_n,\underline{u}_1) - Ce^{-2\lambda_* (x+x_n)}\underline{u}_1
\end{equation}
with initial datum $\kappa \inf_{t  \in \R} \phi_{c'',R,y} (t,x)$. Up to extraction of another subsequence, we can assume that $x_n - \lfloor \frac{x_n}{L} \rfloor L \rightarrow l \in [0,L)$ and (recall that $t_n - \frac{x_n - y}{c''} \rightarrow +\infty$) that $\underline{u}_1 \left( t+ t_n - \frac{x_n-y}{c''} , x; x_n\right)$ converges locally uniformly to an entire solution $\underline{u}_\infty$ of
\begin{equation}\label{eqn:plop2}
\partial_t \underline{u}_\infty = \partial_{xx} \underline{u}_\infty+ \tilde{f}(x+l,\underline{u}_\infty).
\end{equation}
We now prove that $u_\infty \equiv \tilde{p}$. As above, since $c^* >0$ and up to increasing $R$ and $y$, we can assume provided that $\kappa '$ is small enough that $\kappa ' \phi_{0,R/2,x_n}$ is a subsolution of \eqref{eqn:plop}, and also that it lies below~$\kappa \inf_{t  \in \R} \phi_{c',R,y} (t,x)$. It follows that for any $n$, $\underline{u}_1 (\cdot,\cdot;x_n) \geq \kappa ' \phi_{0,R/2,x_n}$. Note that~$s \mapsto \phi_{0,R/2,s}$ is periodic with respect to $s$: this comes from the periodicity of $\tilde{f}$ and the uniqueness of the principal eigenfunction of \eqref{another_eigen} up to multiplication by some factor (here, we always choose it normalized so that~$\| \phi_{0,R/2,s} \|_\infty=1$). By passing to the limit, we get that for any $(t,x) \in \R^2$:
$$\underline{u}_\infty (t,x) \geq \kappa ' \inf_{s \in \R} \phi_{0,R/2,s} (x)$$
where the right hand side is positive on a neighborhood of 0. Therefore, for any $t \in \R$, the function $\underline{u}_\infty (t,x)$ lies above the limit in large time of the solution of \eqref{eqn:plop2} with initial datum $\kappa ' \inf_{s \in \R} \phi_{0,R/2,s} (x)$, which is known to be $\tilde{p}(\cdot +l)$ from the KPP hypothesis on $\tilde{f}$. We conclude that $\underline{u}_\infty (t,x) =\underline{u}_\infty (x) \equiv \tilde{p} (x+l) = \lim_{x \rightarrow +\infty} p (x+x_n)$.

Therefore, $\lim_{n \rightarrow +\infty} u (t_n,x_n) \geq \underline{u}_\infty (0,0) = \lim_{n \rightarrow +\infty} p (x_n)$. As we also know from the comparison principle that $u(t_n,x_n) \leq p(x_n)$ we have finally reached a contradiction, and we can conclude that $u$ spreads to the right with speed $c^*$. This ends the proof of Lemma~\ref{lem:spread}.
\end{proof}

\begin{rmq}
One could easily check from the proof that this result in fact holds as long as $\partial_u f (x,0) \rightarrow \partial_u \tilde{f} (x,0)$ as $x \rightarrow +\infty$, independently of the rate of this convergence. One could also use the characterization of spreading speeds problems using generalized eigenvalues, as in \cite{BerestyckiNadin}, where much more heterogeneous problems where considered, but with the additional assumption that~$f$ is positive between the two stationary states, which is not always satisfied here.
\end{rmq}

\section{Convergence in a close-to-periodic medium}

We now begin the proof of Theorem~\ref{th:CV_ctp}. Let us fix any small $\varepsilon >0$ and show that for $t$ large enough, the difference between $u$ and some shift of $U_{c^*}$ is less than $\varepsilon$. The main idea is the same as in the periodic case. For each $k \in \mathbb{N}$, we again let 
$$t_k=\inf \{t \geq 0 \; | \ u(t_k,kL) = \tilde{p}(0)/2 \},$$
where $\tilde{p}$ is the unique positive and periodic stationary solution of \eqref{eqn:eqRD_lim}. The fact that $t_k$ is well-defined, at least for any large $k$, comes from the fact that $p(kL)\rightarrow \tilde{p}(0) >0$ as $k \rightarrow +\infty$, and the fact that $u$ converges to $p$ in large time from Lemma~\ref{lem:spread}. This even gives us that $t_k \sim \frac{kL}{c^*}$ as~$k \rightarrow +\infty$.

As before, we show that any limit of $u(\cdot +t_k, \cdot + kL)$ is both steeper and less steep than~$U_{c^*}$. As before, the fact that it is less steep is a consequence of \cite{DGM,Nadin??}, where it has been shown that $U_{c^*}$ is steeper than any other entire solution of \eqref{eqn:eqRD_lim}, and of the fact that any limit of $u(\cdot +t_k, \cdot + kL)$ is an entire solution of \eqref{eqn:eqRD_lim}.

To prove the other part, that is that it is steeper than $U_{c^*}$, is far more intricate. We will need to introduce another initial datum $v_0$ less steep than $u_0$, and such that the associated solutions of both the periodic and close-to-periodic problems stay very close from each other. We will then conclude using Section~\ref{sec:periodic}.

\subsection{Choice of $v_0$}

We will introduce here another compactly supported initial datum, namely $v_{0,k}$. Its support will be always chosen to the right of the support of $u_0$, so that $v_0$ is less steep than $u_0$. Moreover, in order to use our intersection number argument, that is Lemma~\ref{lem:steep}, $v_{0,k}$ will be chosen so that $t_k$ is roughly the first time such that $v_k (t_k, kL) = \frac{\tilde{p}(0)}{2}$, and $v_k (t,kL)<\frac{\tilde{p} (0)}{2}$ for any $0 \leq t < t_k$, where $v_k$ is the solution of \eqref{eqn:eqRD} with initial datum $v_{0,k}$. This means that $u(t_k,\cdot)$ will lie above $v_k (t_k,\cdot)$ on the left of $x=kL$, and below on the right. If $v(t_k,\cdot)$ is close to the pulsating traveling wave $U_{c^*}$, this will give the wanted conclusion.

To show this, we will need to prove that $v_k$ stays close to $\tilde{v}_k$ (the solution of the limiting periodic reaction-diffusion equation with initial datum $v_0$) up to the time $t_k$, and that $\tilde{v}_k (t_k,\cdot)$ is already close enough to $U_{c^*}$ (it is known from our first main theorem that $\tilde{v}_k$ converges to $U_{c^*}$ in large time, but here, both the initial datum and the time when we look at the solution depend on $k$).

Roughly speaking, on one hand, the first part requires the support of $v_k$ to be far enough to the right of the domain, so that $f$ is very close to $\tilde{f}$. On the other hand, the second part requires, to apply the same proof we used in the periodic case, that the ratio of the distance between $kL$ and the support of $v_{0,k}$ over the time $t_k$ (needed for the profile of $v_k$ to reach $kL$) is close to $c^*$: in other words, as $\frac{kL}{t_k} \sim c^*$ from Lemma~\ref{lem:spread}, the support of $v_{0,k}$ needs to be relatively close to $x=0$. Therefore, one needs to proceed very carefully in the choice of $v_{0,k}$ to match both requirements simultaneously.\\

As mentioned above, we know that $\frac{kL}{t_k} \sim c^*$. Furthermore, provided that $\varepsilon$ is small enough and using again Lemma~\ref{lem:spread}, there exists $\delta (k) \rightarrow 0$ as $k \rightarrow +\infty$, such that
$$\forall 0 \leq x \leq kL -\delta (k) k \; , \ | u(t,x) - p(x) | \leq \varepsilon ,$$
$$\forall x \geq kL + \delta (k) k \; , \ | u(t,x) | \leq \varepsilon ,$$
$$| kL - t_k c^* | \leq  \delta (k) k.$$
Similarly, fix $R >0$ and look at the solution of the limiting periodic reaction-diffusion equation with initial datum $ \tilde{p} \chi_{(-R,R)}$. There exists $\delta ' (k) \rightarrow 0$ as $k \rightarrow +\infty$ such that the smallest time $t'_k$ when this solution reaches the value $\frac{\tilde{p}(0)}{2}$ at the point $x=kL$ satisfies
$$| kL - t'_k c^* | \leq \delta' ( k) k.$$
We can also assume without loss of generality that $\delta'$ is decreasing with respect to $k$.\\

We let $y_k$ such that
$$y_k \sim \max \{ \sqrt{kL}, \sqrt{|\delta (k)|} k, \sqrt{|\delta '(k/2)|}k \},$$
$$y_k = j_k L \mbox{ with } j_k \in \mathbb{N}.$$
We then define 
$$v_{0,k} = \gamma_k \tilde{p} \times \chi_{(y_k - R, y_k +R)}$$
 with $R$ fixed independently of $k$ and $0<\gamma_k \leq 1$ chosen so that
$$t_k =  \inf \left\{ t >0 \; | \ \tilde{v}_k (t,kL) =\frac{\tilde{p}(0)}{2} \right\}.$$
Recall that $\tilde{v}_k$ denotes the solution of the limiting periodic reaction-diffusion equation with initial datum $v_{0,k}$. The fact that $\gamma_k$ is well-defined follows from the fact that, for $\gamma_k =1$ and $k$ large enough, 
\begin{eqnarray*}
\inf \left\{ t >0 \; | \ \tilde{v}_k (t,kL) =\frac{\tilde{p}(0)}{2} \right\} & = &  t'_{k - \frac{y_k}{L}} = \frac{kL - y_k - \delta' (k-\frac{y_k}{L})(k-\frac{y_k}{L})}{c^*} \\
& \leq &   \frac{kL - y_k + |\delta' (\frac{k}{2})|k}{c^*} \leq t_k - \frac{y_k}{c^*} + o (y_k) < t_k,
\end{eqnarray*}
while this same infimum goes to infinity as $\gamma_k$ gets closer to 0. In a similar fashion, one could actually even show that $\gamma_k \rightarrow 0$ as $k \rightarrow +\infty$.\\

The second property of $y_k$, namely that it is a multiple of the period $L$ is mostly to simplify the computations. The first property implies that $y_k$ is very small compared to $kL$, which will give the fact that $\tilde{v} (t_k ,\cdot)$ will be close to $U_{c^*}$, but is very large compared to $\delta (k) k$, which will allow us to show that $v - \tilde{v}$ does not have time to spread up to $x=kL$ before time $t_k$.

\subsection{Close-to-periodic solution}

In this section, we show the following lemma, which states that $v_k$ and $\tilde{v}_k$ are close from each other at the time $t_k$.

\begin{lem}
There exists $\zeta_1 (k) \rightarrow 0$ as $k \rightarrow +\infty$ such that for any choice of $\gamma_k \leq 1$ and~$x \geq kL - \delta (k) k $, one has
$$| v_k (t_k,x) - \tilde{v}_k (t_k,x) | \leq \zeta_1 (k).$$
\end{lem}

\begin{proof}
Denote by $w (t,x)= v_k (t,x) - \tilde{v}_k (t,x)$ which satisfies
$$
\left\{
\begin{array}{l}
\partial_t w = \partial_{xx} w + f(x,v_k) -\tilde{f} (x,\tilde{v}_k), \vspace{3pt} \\
w(t=0,x) \equiv 0.
\end{array}
\right.
$$
Now note that for any $x$, it follows from the KPP and the close to periodic hypotheses that 
\begin{eqnarray*}|f(x,v_k(t,x)) - \tilde{f} (x,\tilde{v}_k (t,x))| & \leq & |f(x,v_k(t,x)) - \tilde{f} (x,v_k (t,x))| + |\tilde{f}(x,v_k(t,x)) - \tilde{f} (x,\tilde{v}_k (t,x))| \\
&\leq & \min \{ 2 \| \partial_u f(\cdot,0) \|_\infty ,C e^{-2\lambda_* x} \} v_k (t,x) + \partial_u \tilde{f} (x,0) |w(t,x)|.
\end{eqnarray*}
Thus, any nonnegative supersolution $\overline{w}$ of 
$$\partial_t \overline{w} \geq \partial_{xx} \overline{w} + \min \{ 2 \| \partial_u f(\cdot,0) \|_\infty ,C e^{-2\lambda_* x} \} v_k (t,x) + \partial_u \tilde{f} (x,0) \overline{w}$$
is a supersolution for the equation satisfied by $w$. 

Before giving such a $\overline{w}$, let us first estimate $v_k (t,x)$. Let $\lambda_* >0$ be the unique solution of~$\lambda^2 - c^* \lambda = \mu (\lambda)$. Then let $c' >0$ such that $\lambda_*^2 - c' \lambda_* = -\| \partial_u f (\cdot,0) \|_\infty$.
One can check that the following function
$$\overline{v}_k (t,x) := \min \{ p(x) \; , \; e^{\lambda_* (x+R-y_k+c't)} \}$$ 
is a supersolution of \eqref{eqn:eqRD} and is such that $\overline{v}_k (0,\cdot) \geq v_k (0,\cdot)$. Hence, from the comparison principle, $v_k (t,x) \leq e^{\lambda_* (x+R-y_k+c't)}$ for any $t>0$ and $x\in \R$.

Define now $$\overline{w}_1 (t,x):= \max \{C \; , \; 2 \| \partial_u f(\cdot,0) \|_\infty\}  \times \min \{ 1 \; , \; e^{-\lambda_*x}\} \times t e^{\lambda_* (R -y_k +c't)}\geq 0.$$
Let us check that this is indeed a supersolution for the equation satisfied by $w$. First, for any $x  < 0$:
\begin{eqnarray*}
&& \partial_t \overline{w}_1 - \partial_{xx} \overline{w}_1  - \min \{2 \| \partial_u f (\cdot, 0) \|_\infty , Ce^{-2\lambda_* x} \} v_k (t,x) - \partial_u \tilde{f} (x,0) \overline{w}_1 \\
& \geq & \partial_t \overline{w}_1 - 2 \| \partial_u f (\cdot, 0) \|_\infty e^{\lambda_* (x+R-y_k+c't)} - \partial_u \tilde{f} (x,0) \overline{w}_1 \\
& \geq & \partial_t \overline{w}_1 - 2 \| \partial_u f (\cdot, 0) \|_\infty e^{\lambda_* (R-y_k+c't)} - \partial_u \tilde{f} (x,0) \overline{w}_1 \\
& \geq & \partial_t \overline{w}_1 - 2 \| \partial_u f (\cdot, 0) \|_\infty e^{\lambda_* (R-y_k+c't)} - \partial_u \tilde{f} (x,0) \overline{w}_1 \\
& \geq & \frac{\overline{w}_1} {t} - 2 \| \partial_u f (\cdot, 0) \|_\infty e^{\lambda_* (R-y_k+c't)} + c' \lambda_*  \overline{w}_1  - \| \partial_u f (\cdot ,0)\|_\infty \overline{w}_1 \\
& \geq & \lambda_*^2 \overline{w}_1 \geq 0.
\end{eqnarray*}
Besides, for any $x >0$:
\begin{eqnarray*}
&& \partial_t \overline{w}_1 - \partial_{xx} \overline{w}_1  - \min \{2 \| \partial_u f (\cdot, 0) \|_\infty , Ce^{-2\lambda_* x} \} v_k (t,x) - \partial_u \tilde{f} (x,0) \overline{w}_1 \\
& \geq & \partial_t \overline{w}_1 - \partial_{xx} \overline{w}_1 -Ce^{-2\lambda_*x}  e^{\lambda_* (x+R-y_k+c't)} - \partial_u \tilde{f} (x,0) \overline{w}_1 \\
& \geq & \frac{\overline{w}_1}{t} + c'\lambda_*  \overline{w}_1 - \lambda_*^2 \overline{w}_1 - C e^{\lambda_* (-x+R-y_k+c't)} -\| \partial_u \tilde{f} (\cdot,0) \|_\infty \overline{w}_1 \\
& \geq & \frac{\overline{w}_1}{t}  - C e^{\lambda_* (-x+R-y_k+c't)}  \\
& \geq & C e^{\lambda_* (R-y_k +c't)} \times \left( e^{-\lambda_*x} - e^{-\lambda_* x} \right) \geq 0.
\end{eqnarray*}
\begin{rmq}
Note that this is where we need our strong ``close-to-periodic" hypothesis, that is that $f$ converges to a periodic function at least exponentially with some factor larger than $2 \lambda_*$ as~$x \rightarrow +\infty$. It is likely that this assumption could be slightly weakened, although it is not clear up to what extent if using a similar proof.
\end{rmq}

Therefore, it follows from the parabolic comparison principle and the fact that $$w (t=0,\cdot)\equiv 0 \leq \overline{w}_1 (t=0,\cdot),$$ that for any $t>0$ and $x\in \R$:
$$w(t,x) \leq \overline{w}_1 (t,x).$$
However, although small at first, $\overline{w}_1$ then spreads with the speed $c'$ which is larger than $c^*$, which means that it will catch up with $\tilde{v}_k$ before the time $t_k$. Another supersolution is needed.\\

Let us estimate the smallest time $s_k$ such that $\overline{w}_1 (s_k,0) = \| p \|_\infty$. One then has that
$$ s_k e^{\lambda_* ( -y_k +c' s_k)}=\frac{\| p \|_\infty e^{-\lambda_*R}}{\max \{ C, 2 \| \partial_u f (\cdot ,0) \|_\infty \} }.$$
It is straightforward to check that $c' s_k \sim y_k$. We now want to find a supersolution $\overline{w}_2$ to be used from the time $s_k$. Define
$$\overline{w}_2 (t,x) := \| p \|_\infty (1 +C t) \; \frac{ \phi_{\lambda_*} (x)}{\min \phi_{\lambda_*}} e^{-\lambda_* (x-c_*t)}  \geq 0.$$
It is clear from the choice of $s_k$ that $w (s_k , \cdot) \leq \overline{w}_1 (s_k, \cdot) \leq \overline{w}_2 (0,\cdot)$. Denote by $x (t)$, for each time, the smallest point $x$ such that $\overline{w}_2 (t,x) \leq \| p \|_\infty$. It is well-defined, nonnegative and increasing with respect to $t$. It is clear from the definition of $w$ that $w (t,x) \leq \| p \|_\infty$ for all $t$ and $x$. Therefore, one only needs to check that $\overline{w}_2$ is a super-solution of the equation satisfied by $w$ on the domain $\{ (t,x) \; | \; x \geq x(t) \}$. In any point of this domain, one has:
\begin{eqnarray*}
&& \partial_t \overline{w}_2 - \partial_{xx} \overline{w}_2  - \min \{2 \| \partial_u f (\cdot, 0) \|_\infty , Ce^{-2\lambda_* x} \} v_k (t+s_k,x) - \partial_u \tilde{f} (x,0) \overline{w}_2 \\
& \geq & \partial_t \overline{w}_2 - \partial_{xx} \overline{w}_2  -Ce^{-2\lambda_* x} \| p \|_\infty - \partial_u \tilde{f} (x,0) \overline{w}_2 \\
& \geq & c_* \lambda_* \overline{w}_2 + C e^{-\lambda_* (x-ct)} \| p \|_\infty - \lambda_*^2 \overline{w}_2 + 2 \lambda_* \frac{\partial_x \phi_{\lambda_*}}{\phi_{\lambda_*}} \overline{w}_2  - \frac{\partial_{xx} \phi_{\lambda_*}}{\phi_{\lambda_*}}{\overline{w}_2} - C^{-2\lambda_* x}\| p \|_\infty - \partial_u \tilde{f} (x,0) \overline{w}_2\\
& \geq & \overline{w}_2 \left( \mu (\lambda_*) + c_* \lambda_* - \lambda_*^2  \right) + Ce^{-\lambda_* (x-ct)} \| p \|_\infty- Ce^{-2\lambda_* x} \| p \|_\infty\\
& \geq & C \| p \|_\infty ( e^{-\lambda_* x} - e^{-2\lambda_* x}) \geq 0.
\end{eqnarray*}
We conclude that for any $t>0$ and $x \in \R$:
$$w (s_k +t,x) \leq \min \{ \| p \|_\infty , \overline{w}_2 (t,x) \}.$$
Therefore, for any $x \geq kL - \delta (k) k$,
\begin{eqnarray*}
w (t_k, x )& \leq & \overline{w}_2 (t_k - s_k,x) \\
& \leq & \| p \|_\infty (1 + C t_k ) \frac{1}{\min \phi_{\lambda_*}} e^{-\lambda_* (kL - \delta (k) k-c_* (t_k -s_k))}\\
& \leq & \| p \|_\infty (1 + C t_k ) \frac{1}{\min \phi_{\lambda_*}} e^{-\lambda_* (-2 |\delta (k)| k + c_* s_k)} =: \zeta_1 (k).
\end{eqnarray*}
It follows from the fact that $c' s_k \sim y_k$ and the choice of $y_k$ that $\zeta (k) \rightarrow 0$ as $k \rightarrow +\infty$, independently of the choice of $\gamma_k \leq 1$. This ends the proof of our lemma.\end{proof} 

\subsection{Convergence to the pulsating front}

We first state some lemma, which relies on the same proof as in the periodic case:
\begin{lem}
There exists $\zeta_2 (k) \rightarrow 0$ as $k \rightarrow +\infty$ such that for any choice of $\gamma_k \leq 1$ and~$x > kL - \delta (k) k$, one has
$$| \tilde{v}_k (t_k,x) - U_{c^*} (0,x-kL) | \leq \zeta_2 (k).$$
\end{lem}
\begin{proof}
Recall that 
$$t_k = \inf \left\{ t >0 \; | \ \tilde{v}_k (t,kL) = \frac{\tilde{p}(0)}{2} \right\}.$$
As we apply the same method as in the periodic section to the shifted solution $\tilde{v}_k (t,x+y_k)$ of the limiting periodic reaction-diffusion equation, we will omit here the details. Note that, since $\gamma_k \leq 1$ for any $k$, there exits $A$ such that for any $k$, $\tilde{v}_k (0,x+y_k) \leq A e^{-\lambda_\infty x}$.

Let $$\delta_k = \sqrt{2\frac{y_k}{t_k}} \rightarrow 0,$$ 
$$r_k = 2 \lambda^* \delta_k t_k \rightarrow +\infty,$$
where the limits hold as $k \rightarrow +\infty$. Since we already know that $t_k \sim \frac{kL}{c^*} \sim \frac{kL - y_k}{c^*}$, and because we have a uniform bound on the initial datum $\tilde{v}_k (0,\cdot)$, we can use the same method as in Section~\ref{sec:periodic} to get that
$$| \tilde{v}_k (t_k,x) - U_{c^*} (0,x-kL) | \leq \eta_{1,\delta_k} (r_k) + \eta_{2,\delta_k} (r_k) + \epsilon_1 (\delta_k) + \epsilon_2 (k) + \zeta (t_k (r_k))
$$
uniformly with respect to $x \geq kL - \frac{\delta_k^2}{2} t_k = kL-y_k$, where all the terms of the right-hand side have been introduced in Section~\ref{sec:periodic} and still converge to 0 as $k \rightarrow +\infty$ (in Section~\ref{sec:periodic}, we used the fact that~$\delta$ did not depend on~$k$, but one can directly check here the convergence to 0 as $k \rightarrow +\infty$ thanks to our choice of parameters). Lastly, since $\delta (k) k = o(y_k)$, the conclusion of our lemma easily follows.
\end{proof}\\

We can now conclude the proof of Theorem~\ref{th:CV_ctp}. From all the above, we know that $u(t_k,\cdot)$ is steeper than $v(t_k,\cdot)$, which is close to the pulsating traveling wave for any $x \geq kL - \delta (k)k$:
\begin{eqnarray*}
u(t_k,x) & \leq & U_{c^*} (0,x-kL) + \zeta_1 (k) + \zeta_2 (k), \ \  x \geq kL,\\
& \geq & U_{c^*} (0,x-kL) - \zeta_1 (k) - \zeta_2 (k), \ \ kL - \delta (k) k  \leq x \leq kL,
\end{eqnarray*}
where $\zeta_1 (k) + \zeta_2 (k) \rightarrow 0$ as $k \rightarrow +\infty$. On the other hand, we have mentioned that $u (t+t_k,x+kL)$ converges locally uniformly to an entire solution of \eqref{eqn:eqRD_lim}, which is necessarily less steep than any shift of $U_{c^*}$ (as it is the steepest of all entire solutions). It follows that $u(t_k,x+kL)$ converges locally uniformly to $U_{c^*} (0,x)$.

Let us check that this convergence is in fact uniform in the right half space $\{ x \geq \alpha (t_k) \}$, for any $\alpha (t) \rightarrow +\infty$ as $t \rightarrow +\infty$. Recall that $u(t_k,x) \geq p(x)-\varepsilon$ for any $0 \leq x \leq kL - \delta (k)k$. For $k$ large enough, one has that $| \tilde{p} (x) - p (x) | \leq \varepsilon$ for any $x \geq \alpha (t)$. It follows that 
$$p(x) \geq u(t_k,x) \geq p(x) - 2 \varepsilon \ \mbox{ for all } \ \alpha (t_k) \leq x \leq kL - \delta (k) k.$$
From the asymptotics of $U_{c^*}$, there exists $D>0$ such that $U_{c^*} (0,x) \leq \varepsilon$ for any $x \geq D$, and $U_{c^*} (0,x) \geq \tilde{p} (x) - \varepsilon$ for any $x \leq D$. As $u (t_k,x)$ is steeper in the weaker sense above, and stays between $0$ and $p$ from the comparison principle, we get for $k$ large enough that
$$| u(t_k,x+kL) - U_{c^*} (0,x) | \leq 2\varepsilon \ \mbox{ for all } \  x \geq kL +D,$$
$$| u(t_k,x+kL) - U_{c^*} (0,x) | \leq 2\varepsilon \ \mbox{ for all } \ kL - \delta (k) k \leq x \leq kL -D.$$
Lastly, we have seen that $u(t_k,x+kL)$ converges locally uniformly to $U_{c^*} (0,x)$, so that the same inequality as above also holds for $k$ large enough and $kL- D \leq x \leq kL+D$. Since $\varepsilon$ could be chosen arbitrarily small, we conclude that
$$\| u(t_k,x) - U_{c^*} (0,x-kL) \|_{L^\infty (\alpha (t_k),+\infty)} \rightarrow 0$$
as $k \rightarrow +\infty$. One can then easily proceed as in Section~\ref{sec:ending} to get the wanted convergence result. This finally ends the proof of Theorem \ref{th:CV_ctp}.

\end{document}